\documentclass[runningheads]{llncs}
\usepackage[T1]{fontenc}
\usepackage[english]{babel} 
\usepackage[utf8]{inputenc}
\usepackage{mathtools}
\usepackage{amssymb}

\usepackage{amsthm}
\usepackage{authblk}
\usepackage[misc]{ifsym}
\usepackage{hyperref}

\usepackage{dirtytalk}
\usepackage{xcolor}
\usepackage{framed}
\definecolor{shadecolor}{RGB}{153,204,255}

\usepackage{enumerate}

\newtheorem{claimaaa}{Claim}

\newcommand{\rationals}{\ensuremath{{\mathbb{Q}}}}

\newcommand{\defhigh}[1]{\textsc{#1}}

\newcommand{\redsolovay}[1][\le]{\ensuremath{{#1}_{\mathrm{S}} } }
\newcommand{\redsolovaymonotone}[1][\le]{\ensuremath{{#1}_{\mathrm{S}}^{\mathrm{m}} } }
\newcommand{\redsolovaytotal}[1][\le]{\ensuremath{{#1}_{\mathrm{S}}^{\mathrm{tot}} } }

\newcommand{\redsolovayreal}[1][\le]{\ensuremath{{#1}_{\mathrm{S}}^{\mathbb{R}} } }
\newcommand{\notredsolovayreal}[1][\nleq]{\ensuremath{{#1}_{\mathrm{S}}^{\mathbb{R}} } }
\newcommand{\redclopen}[1][\le]{\ensuremath{{#1}_{\mathrm{cL}}^{\mathrm{open}} } }
\newcommand{\redcllocal}[1][\le]{\ensuremath{{#1}_{\mathrm{cL}}^{\mathrm{loc}} } }

\begin{document}

\title{Variants of Solovay reducibility \thanks{supported in part by the ANR project FLITTLA ANR-21-CE48-0023}}

\author{Ivan Titov}
\authorrunning{I.\ Titov}
\institute{Universität Heidelberg, Institut für Informatik, 69120 Heidelberg, Germany \and Université de Bordeaux, CNRS, Bordeaux INP, LaBRI, UMR 5800, F-33400 Talence, France}

\maketitle

\begin{abstract}
Solovay reducibility was introduced by Robert M.\ Solovay~\cite{Solovay-1975} in 1975 in terms of translation functions on rationals. It is a measure of relative approximation speed and thus also of relative randomness of reals. In the area of algorithmic randomness, Solovay reducibility has been intensively studied and several central results on left-c.e.\ reals have been obtained.

Outside of the left-c.e.\ reals, Solovay reducibility is considered to be behaved badly~\cite{Downey-Hirschfeldt-2010}. Proposals for variants of Solovay reducibility that are better suited for the investigation of arbitrary, not necessarily left-c.e. reals were made by Rettinger and Zheng~\cite{Zheng-Rettinger-2004}, and, recently, by Titov~\cite{Titov-2023} and by Kumabe and co-authors~\cite{Kumabe-etal-2020}, \cite{Kumabe-etal-2023}. These variants all coincide with the original version of Solovay reducibility on the left-c.e. reals. Furthermore, they are all defined in terms of translation functions. The latter translate between computable approximations in the case of Rettinger and Zheng, are monotone in the case of Titov, and are functions between reals in the case of Kumabe et al.  

In what follows, we derive new results on the mentioned variants and their relation to each other. In particular, we obtain that Solovay reducibility defined in terms of translation function on rationals implies Solovay reducibility defined in terms of translation functions on reals, and we show that the original version of Solovay reducibility is strictly weaker than its monotone variant.

Solovay reducibility and its variants mentioned so far have tight connections to Martin-Löf randomness, the strongest and most central notion of a random sequence. For the investigation of Schnorr randomness, total variants of Solovay reducibility have been introduced by Merkle and Titov~\cite{Merkle-Titov-2022} in 2022 and, independently, by Kumabe et al.~\cite{Kumabe-etal-2023} in 2024, the latter again via real-valued translation functions. In what follows, we show that total Solovay reducibility defined in terms of rational functions implies total Solovay reducibility defined in terms of real functions.
\end{abstract}


\section{Solovay reducibility and its variants}
We start with reviewing the concept of Solovay reducibility introduced by Solovay~\cite{Solovay-1975} in 1975 as a measure of relative randomness. Its original definition uses the notion of a translation function on rationals, or a $\text{$\mathbb{Q}$-translation}$ function, defined on the left cut of a real. Our notation is standard. All rationals and reals are supposed to be on the interval~$[0,1)$ if not stated otherwise. A left-c.e.\ approximation is a strictly increasing computable approximation.
Unexplained notation can be found in Downey and Hirschfeldt~\cite{Downey-Hirschfeldt-2010}.

\subsection{Solovay reducibility}
\begin{definition}[Solovay, 1975]\label{def:Solovay-reducibility}
The \defhigh{left cut} of a real~$\beta$ is the set~$\mathrm{LC}(\beta)$ of all rationals~$y$ where~$0 \le y < \beta$. 
A $\mathbb{Q}$-\defhigh{translation function} from a real~$\beta$ to a real~$\alpha$ is a partially computable function $g$ from the set~$\rationals \cap [0,1)$ to itself such that for every~$q\in\mathrm{LC}(\beta)$ the value $g(q)$ is defined and fulfills~$g(q)<\alpha$, and it holds that
    \begin{equation}\label{eq:Q-translation-function}
        \lim\limits_{q\nearrow\beta} g(q) = \alpha,
    \end{equation}
    where $\lim\limits_{q\nearrow\beta}$ denotes the left limit.
    
    A real~$\alpha$ is \defhigh{Solovay reducible} to a real~$\beta$, also written as~$\alpha\redsolovay\beta$, if there is a real constant~$c$ and a $\mathbb{Q}$-translation function $g$ from~$\beta$ to~$\alpha$ such that, for all~$q< \beta$, it holds that
    \begin{equation}\label{eq:Solovay-reducibility}
        \alpha - g(q)<c(\beta - q).
    \end{equation}
\end{definition}
A~$\mathbb{Q}$-translation function maps rationals to rationals~--- in contrast to~$\mathbb{R}$-translation functions, to be introduced in the next section, which map reals to reals. We refer to inequality~\eqref{eq:Solovay-reducibility} as \defhigh{Solovay condition} and to the constant~$c$ that occurs in it as \defhigh{Solovay constant}.

\medskip 
Recall that a function~$f$ is \defhigh{Lipschitz continuous} on some interval~$I$ if there exists a constant~$d$, called \defhigh{Lipschitz constant}, such that for every~$p,q\in I$ in the domain of~$f$ it holds that~$|f(q)-f(p)|< d|q-p|$. 

The Solovay condition~\eqref{eq:Solovay-reducibility} resembles a localized version of the defining condition of Lipschitz continuity, where instead of arbitrary pairs of arguments, we consider only pairs with second component~$\beta$. Accordingly, there are close relations between $\mathbb{Q}$-translation functions that are Lipschitz continuous and those that witness Solovay reducibility.
\begin{proposition}\label{rationals-Lipschitz}
    Let~$\alpha$ and~$\beta$ be two reals.
    \begin{enumerate}[(a)]
        \item \label{rationals-Lipschitz:direct}
        If~$g$ is a Lipschitz continuous $\mathbb{Q}$-translation function from~$\beta$ to~$\alpha$, then~$\alpha\redsolovay\beta$ via~$g$.
        \item \label{rationals-Lipschitz:leftce}
        If~$\alpha$ and~$\beta$ are left-c.e.\ and $\alpha\redsolovay\beta$, then $\alpha\redsolovay\beta$ via some Lipschitz continuous $\mathbb{Q}$-translation function. Moreover, the function can be chosen strictly increasing.
    \end{enumerate}
\end{proposition}

\begin{proof}
\begin{enumerate}[(a)]
    \item\label{part-one} The proof is standard and is left to the interested reader.

    \item By~\cite[Proposition~9.1.7]{Downey-Hirschfeldt-2010}, there exist a constant $d>c$ and two left-c.e.\ approximations~$a_0,a_1,\dots$ and~$b_0,b_1,\dots$ of~$\alpha$ and $\beta$, respectively, that fulfill
    \begin{equation}\label{eq:9.1.7}
        a_{n+1} - a_n < d(b_{n+1} - b_n)\quad\text{for every }n.
    \end{equation}
    Then, the function $g$ from $\mathrm{LC}(\beta)$ to $\mathrm{LC}(\alpha)$ defined by
    \[g(q) = a_n + \frac{a_{n+1}-a_n}{b_{n+1}-b_n}(q - b_n)\quad\text{for all rationals }q\in [b_n,b_{n+1})\]
    and $g(q) = a_0$ for all rationals $q\in [0,a_0)$
    is Lipschitz continuous with the Lipschitz constant~$d$ by~\eqref{eq:9.1.7}. Further, $g$ is strictly increasing since the sequences $a_0,a_1,\dots$ and $b_0,b_1,\dots$ are strictly increasing.
    Finally, $g$ fulfills \eqref{eq:Q-translation-function} because, for arbitrarily small distance $\beta - b_n$, we have $0<\alpha - g(q)<d(\beta - b_n)$ for all $q\in [a_n,\alpha)$.

    Therefore, $g$ is a Lipschitz continuous strictly increasing $\mathbb{Q}$-translation function from $\beta$ to $\alpha$, hence, by the statement~\eqref{part-one} of the current proposition, witnesses~$\alpha\redsolovay\beta$.
\end{enumerate}
\end{proof}

\medskip

In 2022, Merkle and Titov introduced~\cite{Merkle-Titov-2022} a version of Solovay reducibility via a totally defined $\mathbb{Q}$-translation function.
\begin{definition}[Merkle, Titov, 2022]
    A real~$\alpha$ is \defhigh{total Solovay reducible} to a real~$\beta$, written $\alpha\redsolovaytotal\beta$, if $\alpha\redsolovay\beta$ via a $\mathbb{Q}$-translation function $g$ which is totally computable on~$[0,1)$.
\end{definition}
\subsection{Monotone Solovay reducibility}

In 2024, Titov~\cite{Titov-2023} proposed the following monotone variant of Solovay reducibility, which coincides with~$\redsolovay$ on the set on left-c.e.\ reals.
\begin{definition}
    A real~$\alpha$ is~\defhigh{monotone Solovay reducible} to a real~$\beta$, written~$\alpha\redsolovaymonotone\beta$, if~$\alpha$ is~\defhigh{Solovay reducible} to~$\beta$ via a $\mathbb{Q}$-translation function that is nondecreasing.
\end{definition}

By definition, Solovay reducibility is implied by its monotone variant. In the remainder of this section, we demonstrate that the reverse implication does not hold in general, i.e., $\redsolovaymonotone$ is strictly stronger than~$\redsolovay$. For a start, we state somewhat technical observations about the behavior of monotone $\mathbb{Q}$-translation functions, their proofs are left to the interested reader.
\begin{lemma}\label{monotone-function-left-and-right-to-limit}
    Every monotone $\mathbb{Q}$-translation function $g$ from a real $\beta$ to a real $\alpha$ satisfies the following properties.
    \begin{enumerate}[(a)]
        \item \label{monotone-function-right-to-limit}
        For every $q\geq\beta$ such that $g(q)$ is defined, we have $g(q)\geq\beta$.
        \item \label{monotone-function-left-to-limit}
        If $q_0,q_1,\dots$ is a sequence of rationals in $[0,\beta)$ such that $\lim\limits_{n\to\infty}g(q_n) = \alpha$, then $\lim\limits_{n\to\infty}q_n = \beta$.
    \end{enumerate}
\end{lemma}

The next proposition shows that the monotone $\mathbb{Q}$-translation function can be only from a left-c.e.\ to a left-c.e.\ or from a nonleft-c.e.\ to a nonleft-c.e.\ real.

\begin{proposition}\label{leftce-standalone-equivalence-monotone}
    Let $\alpha,\beta\in \mathbb{R}$, where $\alpha$ is left-c.e. Then the following statements are equivalent:
    \begin{enumerate}[(i)]
        \item $\beta$ is left-c.e.;
        \item there exists a nondecreasing $\mathbb{Q}$-translation function from $\alpha$ to $\beta$;
        \item there exists a nondecreasing $\mathbb{Q}$-translation function from $\beta$ to $\alpha$.
    \end{enumerate}
\end{proposition}

\begin{proof}
    Let $a_0,a_1,\dots$ where $a_0 = 0$ be a left-c.e.\ approximation of~$\alpha$.
    
    \begin{itemize}
        \item $(i)\implies(ii)$ and $(i)\implies(iii)$. In case $\beta$ is left-c.e., let $b_0,b_1,\dots$ where $b_0 = 0$ be a left-c.e.\ approximation of~$\beta$. Then the functions
        \[g(q) = a_{\min\{m: b_m\geq q\}}\makebox[4em]{and}h(q) = b_{\min\{m: a_m\geq q\}}\]
        are nondecreasing translation functions from~$\beta$ to $\alpha$ and from~$\alpha$ to $\beta$, respectively.

        \item $(ii)\implies(i)$. If there exists a nondecreasing $\mathbb{Q}$-translation function $g$ from $\alpha$ to $\beta$, then the sequence~$g(a_0),g(a_1),\dots$ is a nondecreasing computable approximation of~$\beta$, hence~$\beta$ is left-c.e.

        \item $(iii)\implies(i)$. If there exists a nondecreasing $\mathbb{Q}$-translation function $g$ from $\beta$ to $\alpha$, then, for a fixed computable enumeration $q_0=0,q_1,\dots$ of $\mathrm{dom}(g)$, we define $i_0,i_1,\dots$ step-wise: at step $0$, we set $i_0 = 0$; at step $n+1$, we search for an index~$i$ such that 
        \begin{equation}\label{eq:monotone-tranlsation-from-beta-to-alpha}
            q_{i_n} < q_i\text{, } g(q_i)>a_{n},\text{ and there is an index }j>i_n\text{ such that }g(q_i)<a_j
        \end{equation}
        and set~$i_{n+1} = i$.

        Then every step terminates since, for every $n$, all rationals in the interval $(q_{i_n},\beta)\cap(d_n,\beta)$, where $d_n$ is some point in $(0,\beta)$ such that $g(d_n) \in (a_{i_n},\alpha)$, fulfill~\eqref{eq:monotone-tranlsation-from-beta-to-alpha}. Therefore, the sequence $q_{i_0},q_{i_1},\dots$ is infinite and computable.
        It is increasing by $q_{i_n} < q_{i_{n+1}}$.
        Further, for every $j$, we have $q_{i_n}< a_j < \alpha$ for some $j$, so, by Lemma~\ref{monotone-function-left-and-right-to-limit}\eqref{monotone-function-right-to-limit}, the sequence $q_{i_0},q_{i_1},\dots$ lies in~$[0,\beta)$.
        Finally, it fulfills $g(q_{i_{n+1}}) > a_{n}$ for every $n$, hence the sequence $g(q_{i_0}),g(q_{i_1}),\dots$ tends to $\alpha$; and thus, by Lemma~\ref{monotone-function-left-and-right-to-limit}\eqref{monotone-function-left-to-limit}, the sequence $q_{i_0},q_{i_1},\dots$ tends to~$\beta$. Therefore, it is a left-c.e.\ approximation of $\beta$, hence $\beta$ is left-c.e.
    \end{itemize}   
\end{proof}

\begin{corollary}\label{leftce-standalone-monotone}
    The set of left-c.e.\ reals is $\redsolovaymonotone$-closed downwards and upwards~in~$\mathbb{R}$.
\end{corollary}

In contrast, the set of left-c.e.\ reals is not $\redsolovay$-closed downwards, as we will see in the next proposition.

\begin{proposition}\label{eq:1-solovay-reducible-to-nonleft-ce}
    There exist a nonleft-c.e.\ real~$\beta$ such that~$1\redsolovay\beta$.
\end{proposition}
\begin{proof}
    We fix an enumeration~$q_1,q_2,\dots$ of all rationals in the unit interval and define a computable test~$I_n:=[q_n, q_n+2^{-2n}]$.
    Since~$\sum_{n=1}^{\infty}\mu(I_n) = \frac{1}{3}$, while the set of all left-c.e.\ reals has the Lebesgue measure $0$ because there are only countably many such reals, we can fix a nonleft-c.e.\ real~$\beta\notin\bigcup_{i=1}^{\infty}I_n$.

    On the one hand, the function $g$ defined by
    \begin{equation*}\label{eq:define-nonmonotone-translation-function-for-separation}
        g(q_n) = 1 - 2^{-2n}
    \end{equation*}
    is a $\mathbb{Q}$-translation function from~$\beta$ to~$1$.
    Moreover, for every $q_n<\beta$, we easily obtain from $\beta\notin I_n$ that $\beta > q_n+\mu(I_n) = q_n + 2^{-2n}$, and thus
    \[1 - g(q_n) = 1 - (1 - 2^{-2n}) = 2^{-2n} \in (0,\beta - q_n).\]
    Therefore we obtain that $1\redsolovay\beta$ with the constant $1$ via $g$.
\end{proof}

On the other hand, the existence of a nondecreasing $\mathbb{Q}$-translation function from the real~$\beta$ defined in the latter proof to~$1$ would imply by Proposition~\ref{leftce-standalone-equivalence-monotone} that $\beta$ is left-c.e., which cannot be true. 
Therefore, Proposition~\ref{eq:1-solovay-reducible-to-nonleft-ce} allows us to separate $\redsolovaymonotone$ and $\redsolovay$ on $\mathbb{R}$.
\begin{corollary}\label{monotone-is-strictly-stronger}
    There exist two reals~$\alpha$ and $\beta$ such that $\alpha\redsolovay\beta$ but $\alpha\nleq_{\mathrm{S}}^{\mathrm{m}}\beta$.
\end{corollary}

\subsection{Computable functions on reals}
\begin{definition}
    A sequence $q_0,q_1,\dots$ of rationals is called \defhigh{effective approximation} if it fulfills~$|q_n-q_{n+1}|<2^{-n}$ for every~$n$.
\end{definition}

Since every effective approximation~$q_0,q_1,\dots$ is a Cauchy sequence, it converges to some limit point~$x\in \mathbb{R}$. The following properties of effective approximations and their limits can be obtained straightforwardly.
\begin{lemma}\label{effective-distance-to-limit}
    \begin{enumerate}[(a)]
        \item\label{effective-distance-to-limit:direct}
        If $q_0,q_1,\dots$ is an effective approximation, then there exists a real~$x$ such that
        \begin{equation}\label{eq:effective-distance-to-limit:direct}
            |x - q_n| < 2^{-(n-1)}\text{ for all }n\in \mathbb{N}.
        \end{equation}
        In particular, $x = \lim\limits_{n\to\infty}q_n$. In that case, we also say that $q_0,q_1,\dots$ is an~\defhigh{effective approximation of }$x$.
        \item\label{effective-distance-to-limit:inverse}
        Let $q_0,q_1,\dots$ be an approximation of a real~$x$ that satisfies
        \begin{equation}\label{eq:effective-distance-to-limit:inverse}
            |q_{n+1} - q_n| < 2^{-n}\text{ for all }n\in \mathbb{N}.
        \end{equation}
        Then $q_0,q_1,\dots$ is an effective approximation of~$x$.
    \end{enumerate}
    
\end{lemma}

Obviously, a real~$x$ is computable iff there exists a computable effective approximation of~$x$.

The class of Turing machines that, using an infinite sequence of finite strings (in our case, encoded rationals) as an oracle, returns another sequence of finite strings was formally defined by Weihrauch~\cite[Chapter~2]{Weihrauch-2000} under the name~``Turing machines of Type~2''. In what follows, we give a formal definition of a computable real function using a notion of a Turing machine of Type 2 specified for the sequences of rationals.

\begin{definition}[Weihrauch, 2000]
        A \defhigh{Turing Machine} $M$ \defhigh{of Type 2} is an oracle Turing machine that, for every oracle $(p_0,p_1,\dots)$, where $p_0,p_1,\dots$ are (appropriately finitely encoded) rationals, produces either an infinite sequence of rationals~$(q_0,q_1,\dots)$; in this case, we
        write $M^{(p_0,p_1,\dots)}\downarrow = (q_0,q_1,\dots)$ and
        say that~$M$ \defhigh{returns the sequence} $(q_0,q_1,\dots)$ \defhigh{from the input} $(p_0,p_1,\dots)$; or a finite set of rationals $(q_0,q_1,\dots,q_n)$; in the latter case, we say that $M^{(p_0,p_1,\dots)}$~\defhigh{is undefined}.

        If, for an effective approximation $(p_0,p_1,\dots)$ given as an oracle, $M$ returns at least $n+1$ first elements $(q_0,\dots,q_n)$ of some effective approximation  $(q_0,q_1,\dots)$, we say that \defhigh{the computation of }$M^{(p_0,p_1,\dots)}$ \defhigh{with tolerance }$2^{-(n-1)}$ \defhigh{halts} and call $p_n$ \defhigh{result of} $M^{(q_0,q_1,\dots)}$ \defhigh{with tolerance} $2^{-(n-1)}$.
        
        A real function~$f$ from some subset of~$\mathbb{R}$ to~$\mathbb{R}$ is \defhigh{computable} on some set~${X\subseteq \mathrm{dom}(f)}$ if there exists an oracle Turing machine $M$ such that, for every $x\in X$ and every effective approximation $p_0,p_1,\dots$ that converges to~$x$,
        $M^{(p_0,p_1,\dots)}$ returns an effective approximation $q_0,q_1,\dots$ of~$f(x)$.
\end{definition}

By~\cite[Theorem~4.3.1]{Weihrauch-2000}, computability of a real function implies its continuity.
\begin{proposition}
    If a real function~$f$ is computable on some interval~$(a,b)$, then~$f$ is continuous on~$(a,b)$.
\end{proposition}

\subsection{Variants of Solovay reducibility defined via translation function on reals}
According to Proposition~\ref{rationals-Lipschitz}(\ref{rationals-Lipschitz:direct}), on the left-c.e.\ reals, the Solovay reducibility is equivalent to the Solovay reducibility via a Lipschitz continuous $\mathbb{Q}$-translation function. In 2020, Kumabe, Miyabe, Mizusawa, and Suzuki proposed~\cite[Definition~9]{Kumabe-etal-2020} a new type of reducibility by replacing the Lipschitz continuous translation functions on rationals in the latter characterization by the translation functions on reals, without requiring for the reals~$\alpha$ and~$\beta$ to be left-c.e.

In what follows, we give the formal definition of this reducibility denoted by Kumabe et al.\ as ``L2'' under the more intuitive name~``real Solovay reducibility''.
\begin{definition}[Kumabe et al., 2020]\label{define-real-Solovay}
    An \defhigh{$\mathbb{R}$-translation function} from a real~$\beta$ to another real~$\alpha$ is a real function which is computable on the interval~$[0,\beta)$ and maps it to the interval~$[0,\alpha)$ such that
    \begin{equation}\label{eq:R-translation-function}
        \lim\limits_{x\nearrow\beta}f(x) = \alpha.
    \end{equation}
    A real~$\alpha$ is~\defhigh{real Solovay reducible} to a real~$\beta$, written $\alpha\redsolovayreal\beta$, if there exists a Lipschitz continuous $\mathbb{R}$-translation function from~$\beta$ to~$\alpha$.
\end{definition}

Note that, between every two left-c.e.\ reals, there exist both an~$\mathbb{R}$-translation function and a $\mathbb{Q}$-translation function.
\begin{proposition}\label{Existenzsatz}
    If $\alpha$ and $\beta$ are left-c.e.\ reals, then there exists an~$\mathbb{R}$-translation function from~$\beta$ to~$\alpha$, whose restriction on $\mathbb{Q}$ is a $\mathbb{Q}$-translation function from~$\beta$ to~$\alpha$.
\end{proposition}

\begin{proof}
    Fixing two left-c.e.\ approximations~$a_0,a_1,\dots\nearrow\alpha$ and~$b_0,b_1,\dots\nearrow\beta$, we can construct an~$\mathbb{R}$-translation function~$f$ from~$\beta$ to~$\alpha$ by setting $f(x) = a_0$ for all $x<b_0$ and
    \begin{equation*}
        f(x) = a_n + \frac{x-b_n}{b_{n+1} - b_n}(a_{n+1} - a_n)\quad\text{if there exists such }n\text{ that }b_n \leq x < b_{n+1}.
    \end{equation*}
    It is easy to see that, if $x$ is rational, then $f(x)$ is rational too, so $f|_{\mathbb{Q}}$ is a $\mathbb{Q}$-translation function from $\beta$ to $\alpha$.
\end{proof}

In contrast to Solovay reducibility (see Corollary~\ref{monotone-is-strictly-stronger}), the additional requirement for the $\mathbb{R}$-translation function in the latter definition to be nondecreasing (Kumabe et al.\ denoted~\cite[Definition~9]{Kumabe-etal-2020} the resulting reducibility~``L1'') does not induce a strictly stronger reducibility on~$\mathbb{R}$ than~$\redsolovayreal$, as we will see in the next proposition, which can be shown by essentially the same proof as~\cite[Corollary~6.2.5]{Weihrauch-2000} (computability of the maximum operator).

\begin{proposition}\label{real-monotonization}
    If $f$ is a $\mathbb{R}$-translation function from a real~$\beta$ to another real~$\alpha$, then
    \[\Tilde{f}(x) = \max\{f(y):y\in [0,x]\}\]
    is a monotone nondecreasing $\mathbb{R}$-translation function from~$\beta$ to~$\alpha$.
    $\Tilde{f}$ can be computed by a Machine $M$ of Type 2 that satisfies the following property:
    \begin{enumerate}[(A)]
        \item\label{property-P}
        For every real $x<\beta$, every rational $q>\beta$, every effective approximation $p_0,p_1,\dots$ of $x$, and every $n\in \mathbb{N}$, if the computation of $M^{(q,q,\dots)}$ with tolerance $2^{-n}$ halts, and its result is some rational $r$, then $r$ fulfills the inequality $r > \Tilde{f}^{(n)}(x) - 2^{-(n-1)}$, where $\Tilde{f}^{(n)}(x)$ is the result of the computation of $M^{(p_0,p_1,\dots)}$ with the tolerance $2^{-n}$ (which obviously fulfills $|f(x)-\Tilde{f}^{(n)}(x)|<2^{-n}$; so, in particular, $r>f(x) - 2^{-(n-2)}$).
    \end{enumerate}
    Moreover, if~$\alpha\redsolovayreal\beta$ via~$f$, then~$\alpha\redsolovayreal\beta$ via~$\Tilde{f}$ with the same Lipschitz constant.
\end{proposition}

\begin{corollary}\label{eq:real-reducibility-is-monotonizable}
    If~$\alpha\redsolovayreal\beta$ for two reals~$\alpha$ and~$\beta$, then $\alpha\redsolovayreal\beta$ via a nondecreasing $\mathbb{R}$-translation function.
\end{corollary}

The following property of monotone $\mathbb{R}$-translation functions, which is similar to Lemma~\ref{monotone-function-left-and-right-to-limit}\eqref{monotone-function-left-to-limit} for $\mathbb{Q}$-translation functions, can be easily proved.
\begin{lemma}\label{real-monotone-function-left-to-limit}
    If $f$ is a monotone $\mathbb{R}$-translation function from a real $\beta$ to a real $\alpha$ and $x_0,x_1,\dots$ is a sequence of reals in $[0,\beta)$ such that $\lim\limits_{n\to\infty}f(x_n) = \alpha$, then $\lim\limits_{n\to\infty}x_n = \beta$.
\end{lemma}

Similarly as for $\redsolovaymonotone$, we can show the $\redsolovayreal$-closure upwards and downwards of left-c.e.\ reals.

\begin{proposition}\label{leftce-standalone-equivalence-real}
    Let~$\alpha, \beta\in\mathbb{R}$ where~$\alpha$ is left-c.e. Then the following statements are equivalent:
    \begin{enumerate}[(i)]
        \item $\beta$ is left-c.e.;
        \item there exists an $\mathbb{R}$-translation function from~$\alpha$ to~$\beta$;
        \item there exists an $\mathbb{R}$-translation function from~$\beta$ to~$\alpha$.
    \end{enumerate}
\end{proposition}

\begin{proof}
Let $a_0,a_1,\dots$ where $a_0 = 0$ be a left-c.e.\ approximation of~$\alpha$.
\begin{itemize}
    \item $(i)\implies(ii)$ and $(i)\implies(iii)$.
    Directly by Proposition~\ref{Existenzsatz}.

    \item $(ii)\implies(i)$.
    If there exists an $\mathbb{R}$-translation function from~$\alpha$ to~$\beta$, then, by Proposition~\ref{real-monotonization}, we can fix a nondecreasing $\mathbb{R}$-translation function $f$ from~$\alpha$ to~$\beta$.
    Then we compute a left-c.e.\ approximation of $\beta$ step-wise: starting with~$i_{-1} = -1$ and $m_{-1} = 0$, in every step~$n\geq 0$, we approximate the values of $f(a_i)$ for all~$i>i_{n-1}$ effectively until we find a natural~$m>m_{n-1}$ and two indices~$i,j$ where $i_{n-1} <i <j$ such that, for the approximated with tolerance $2^{-m}$ values~$\Tilde{f}(a_i)$ and~$\Tilde{f}(a_j)$ of~$f(a_i)$ and $f(a_j)$, respectively, the inequality
    \[\Tilde{f}(a_i)+2^{-m} < \Tilde{f}(a_j)-2^{-m}\]
    holds (so, by $\Tilde{f}(a_j)<f(a_j) + 2^{-m}< \alpha + 2^{-m}$, we can guarantee that $\Tilde{f}(a_i)<\alpha$), and set~$i_n = i$, $m_n = m$, and $b_{n} = \Tilde{f}(a_{i_n})$. Then the sequence~$b_0,b_1,\dots$ is a left-c.e.\ approximation of~$\beta$.
    
    \item $(iii)\implies(i)$. The case $\beta\in\mathbb{Q}$ is trivial; so, in what follows, assume $\beta\notin\mathbb{Q}$.
    
    If there exists an $\mathbb{R}$-translation function from~$\beta$ to~$\alpha$, then, by Proposition~\ref{real-monotonization}, we can fix a nondecreasing $\mathbb{R}$-translation function $f$ from~$\beta$ to~$\alpha$ and a Turing machine $M$ of Type 2 that computes $f$ and fulfills~\eqref{property-P}.
    Then we compute a left-c.e.\ approximation of $\beta$ step-wise: starting from some fixed $b_{-1} := 0 < \beta$ and $j_{-1} = -1$,
    in every step~$n> 0$, we search for a rational $b > b_{n-1}$, a natural $m$, and an index $j> j_{n-1}$ such that the computation of~$M^{(b,b,\dots)}$ with tolerance $2^{-m}$ halts, and its result $r$ (in case $b<\beta$, $r$ we have~$|f(b)-r| < 2^{-m}$ by definition) fulfills
    \begin{equation}\label{eq:tolerance-if-fine}
        a_{j_{n-1}}+2^{-m} < r - 2^{-m} \makebox[4em]{and} r + 2^{-(m-2)} < a_{j}-2^{-(m-2)},
    \end{equation}
    and set $b_n = b$, $m_n = m$, $r_n = r$, and $j_n = j$.
    
    Step $n$ terminates and yields a rational $b_n<\beta$ due to the following argumentation: \begin{itemize}
        \item assuming by induction hypothesis that $b_{n-1}<\beta$, the first inequality in~\eqref{eq:tolerance-if-fine} is satisfied for all rationals $r=f^{(m_1)}(b)$ where $f^{(m_1)}(b)$ is the result of the computation of $M^{(b,b,\dots)}$ with the tolerance $2^{-m_1}$ (which fulfills $|f(b) - \Tilde{f}^{(m_1)}(b)| < 2^{-m_1}$ since $b<\beta$) for all rationals $b$ in some left neighborhood of $\beta$ (maybe intersected with the interval $[b_{n-1},\beta)$, which is nonempty by induction hypothesis) and all sufficiently large $m_1$. For every such $r$, all but finitely many indices $j$ fulfill $f(b) < a_j$, hence there exists a natural $m_2$ such that $r + 2^{-(m_2-2)} < a_j - 2^{-(m_2-2)}$. We conclude the termination proof by choosing $m> \max\{m_1,m_2\}$.

        \item If, for some $m$ and $j$, there exists a rational $b>\beta$ such that the computation of~$M^{(b,b,\dots)}$ with tolerance $2^{-m}$ halts, and its result $r$ fulfills~\eqref{eq:tolerance-if-fine}, then we obtain from the right part of~\eqref{eq:tolerance-if-fine} that $r + 2^{-(m-2)} < \alpha - 2^{-(m-2)}$,
        hence $r < \alpha - 2^{-(m-3)}$.
        By $\lim\limits_{x\nearrow\beta}f(x) = \alpha$, there exists a real $x<\beta$ such that $r<f(x) - 2^{-(m-2)}$, but this contradicts to the property~\eqref{property-P} applied on $x$ and $b$.
        Finally, from $b\in\mathbb{Q}$ and $\beta\notin\mathbb{Q}$, we obtain that $b\neq \beta$; thus, the only possible case is $b<\beta$.
    \end{itemize}
    Therefore, we have constructed an infinite sequence $b_0<b_1<\dots<\beta$, hence, for every $n$, the definition of $r_n$ and the construction of $M$ imply that
    \begin{equation}\label{eq:r-left-from-beta}
        |f(b_n)-r_n| < 2^{-m_n}.
    \end{equation}
    
    From~\eqref{eq:tolerance-if-fine} and~\eqref{eq:r-left-from-beta}, we easily obtain that
    \begin{equation}\label{eq:b_n-increasing}
        a_{j_{n-1}} < f(b_n) < a_{j_n}.
    \end{equation}
    Thus, the sequence $b_0,b_1,\dots$ lies in~$[0,\beta)$ by~\eqref{eq:r-left-from-beta} and fulfills~$\lim\limits_{n\to\infty}f(b_n) = \alpha$ by~\eqref{eq:b_n-increasing}, hence, by Lemma~\ref{real-monotone-function-left-to-limit}, it converges to $\beta$.
    Since the sequence $b_0,b_1,\dots$ is strictly increasing by~\eqref{eq:r-left-from-beta}, it is a left-c.e.\ approximation of~$\beta$.
\end{itemize}
\end{proof}

\begin{corollary}\label{leftce-standalone-real}
    The set of left-c.e.\ reals is $\redsolovayreal$-closed downwards and upwards.
\end{corollary}

In 2024, Kumabe, Miyabe, and Suzuki introduced~\cite[Chapter~5]{Kumabe-etal-2023} a slightly modified version of~$\redsolovayreal$ by omitting the requirement for the $\mathbb{R}$-translation function $f$ in Definition~\ref{define-real-Solovay} to fulfill~$f(x)<\alpha$ for all~$x<\beta$. Kumabe et al.\ also examined the totalized variant~$\redcllocal$ of the latter reducibility.

\begin{definition}[Kumabe et al., 2024]
    A \defhigh{weakly $\mathbb{R}$-translation function} from a real~$\alpha$ to another real~$\beta$ is a real function which is computable on the interval~$[0,\beta)$ and satisfies~\eqref{eq:R-translation-function}.
    
    A real~$\alpha$ is \defhigh{computably Lipschitz reducible} to a real~$\beta$ \defhigh{on a c.e.\ open interval}, written $\alpha\redclopen\beta$, if there exists a Lipschitz continuous weakly $\mathbb{R}$-translation function~$f$ on reals from~$\beta$ to~$\alpha$.
    
    $\alpha$ is \defhigh{computably Lipschitz reducible} to a real~$\beta$ \defhigh{locally}, written $\alpha\redcllocal\beta$, if there exists a real~$b>\beta$ and a Lipschitz continuous function~$f$ on reals, which is computable on~$[0,b]$ and fulfills~$f(\beta) = \alpha$.
\end{definition}

By~\cite[Theorem 1]{Kumabe-etal-2020} and by~\cite[Observation 5.3]{Kumabe-etal-2023}, respectively, the reducibilities~$\redsolovayreal$ and~$\redclopen$ are equivalent to the Solovay reducibility~$\redsolovay$ on the set of left-c.e.\ reals.

In Section~\ref{section:main-result}, we prove that $\redclopen$ is implied by $\redsolovay$ and~$\redcllocal$ is implied by~$\redsolovaytotal$ on~$\mathbb{R}$ and that $\redsolovayreal$ is implied by $\redsolovay$ on all but computable reals.


\section{Reducibilities~$\redclopen$ and~$\redcllocal$ as measures of relative Martin-Löf and Schnorr randomness, respectively}
On the one hand, it is easy to see that every computable real~$\alpha$ is $\redcllocal$-reducible and thus also~$\redclopen$-reducible to every further real $\beta$ via the weakly $\mathbb{R}$-translation function $f(x) = \alpha$ which is computable (as a real function) and totally defined on the unit interval.

On the other hand, if a real~$\alpha$ is $\redclopen$-reducible to a computable real~$\beta$, then we can easily compute~$\alpha$ as well; the formal proof is left to the reader as an exercise. The latter two observations imply that the least ~$\redclopen$- and~$\redcllocal$-degrees on~$\mathbb{R}$ both contain exactly all computable reals.

\begin{proposition}\label{computable-form-least-degree}
    The computable reals form the least degree on~$\mathbb{R}$ relative to both~$\redclopen$ and~$\redcllocal$.
\end{proposition}

\medskip

The $\redsolovay$-closure upwards of Martin-Löf random reals has been proved by Solovay himself~\cite{Solovay-1975};
the $\redsolovaytotal$-closure upwards of Schnorr random reals has been demonstrated by Merkle and Titov~\cite[Corollary 2.10]{Merkle-Titov-2022}.

In what follows, we prove the same closures for the reducibilities~$\redclopen$ and~$\redclopen$, respectively.
\begin{proposition}
    \begin{enumerate}[(a)]
        \item The set of Martin-Löf random reals is $\redclopen$-closed upwards.
        \item The set of Schnorr random reals is $\redcllocal$-closed upwards.
    \end{enumerate}
\end{proposition}

\begin{proof}
\begin{enumerate}[(a)]
    \item Let~$\alpha$ and $\beta$, where $\beta$ is Martin-Löf nonrandom, be two reals that fulfill~$\alpha\redclopen\beta$ with a constant~$c$ via a real function~$f$
    computed by a Turing machine~$M$ of Type 2.
    In particular, for every rational~$q<\beta$, $f$ satisfies 
    \begin{equation}\label{eq:f-fulfills-Solovay-distance-property}
        |\alpha - f(q)\downarrow| < c(\beta - q).
    \end{equation}
    We define further a two-argument-function~$g:\mathbb{Q}|_{[0,1)}\times \mathbb{N}\to\mathbb{Q}$ by setting
    \begin{equation}\label{eq:define-g}
        g(q,m)\text{ is the result of }M^{(q,q,\dots)}\text{ with tolerance }2^{-m}\text{ if defined}
    \end{equation}
    Then, for every rational~$q<\beta$ and natural~$m$, we know by the choice of~$M$ that $|f(q)\downarrow - g(q,m)| < 2^{-m}$, hence we obtain by~\eqref{eq:f-fulfills-Solovay-distance-property} that
    \begin{equation}\label{eq:g-fulfills-Solovay-distance-property}
         |\alpha - g(q,m)\downarrow| \leq c(\beta - q) + 2^{-m}
    \end{equation}
    for every $q<\beta$ and $m\in\mathbb{N}$.
    
    Let further~$q_0,q_1,\dots$ be a standard enumeration of rationals on~$[0,1)$.
    
    Since~$\beta$ is Martin-Löf nonrandom, we can further fix a Solovay test~$S_0,S_1,\dots$ that fails on~$\beta$, where we denote $S_n = [l_n,r_n]$ of length $d_n$ for every~$n$. In particular, this test has a finite measure:
    \begin{equation}
        L := \sum_{n\in\mathbb{N}}\mu(S_n) = \sum_{n\in\mathbb{N}}d_n <\infty.
    \end{equation}
    Then we prove that~$\alpha$ should be Martin-Löf nonrandom as well by constructing the following Solovay test:
    for every natural~$n$, compute~$g(l_n,n)$ and, if the computation halts, set 
    \begin{equation}\label{eq:define-T_n}
        T_n = [g(l_n,n) - (cd_n + 2^{-n}), g(l_n,n) + (cd_n + 2^{-n}) ].
    \end{equation}
    Then, the test $T_0,T_1,\dots$ that contains all defined~$T_n$ is computable and has a finite measure by
    \begin{equation*}
        \sum_{n\in\mathbb{N}}\mu(T_n) \leq \sum_{n\in\mathbb{N}} (2^{-n+1} + 2c d_n) = \sum_{n\in\mathbb{N}}2^{-n+1} + 2c\sum_{n\in\mathbb{N}} d_n = 4 + 2cL <\infty.
    \end{equation*}
    Further, for every of infinitely many $n$ such that~$\beta \in S_n$,
    we have~$l_n < \beta < r_n$, hence, by~\eqref{eq:g-fulfills-Solovay-distance-property}, we have~$g(l_n,n)\downarrow$ (hence~$T_n$ is defined) and
    \begin{equation}\label{eq:T_n-contains-alpha}
        |\alpha - g(l_n,n)| \leq c(\beta - l_n) + 2^{-n} < c(r_n - l_n) + 2^{-n} = cd_n + 2^{-n},
    \end{equation}
    which implies that~$\alpha\in T_n$. Therefore, the test~$T_0,T_1,\dots$ that contains all defined~$T_n$ is a Solovay test that fails on~$\alpha$, hence $\alpha$ is Martin-Löf nonrandom.
    
    \item If, in the latter proof, $f$ is computable on some interval $[0,b]$ with $b>\beta$ (i.e.\ $\alpha\redcllocal\beta$ via~$f$), then we can fix some rational $\rho\in (\beta,b)$. Obviously, $f$ is computable on $[0,\rho]$, and the function~$g$ defined as in~\eqref{eq:define-g} fulfills~\eqref{eq:g-fulfills-Solovay-distance-property} for every~$q<\rho$ and~$m\in\mathbb{N}$.

    If, additionally, $S_0,S_1,\dots$ is a total Solovay test that fails on~$\beta$ (which witnesses the Schnorr nonrandomness of $\beta$ by~\cite[Theorem~2.5]{Downey-Griffiths-2002}), then we can modify this test by replacing every~$S_n$ by~$S_n\cap [0,\rho]$; the resulting test is, again, a total Solovay test that fails on~$\beta$ and has a computable measure:
    \begin{equation}
        L_{comp} := \sum_{n\in\mathbb{N}}\mu(S_n) = \sum_{n\in\mathbb{N}}d_n < \infty.
    \end{equation}
    
    For every natural $n$ the interval~$T_n$ as in~\eqref{eq:define-T_n} is well-defined since we know from~$l_n < \rho$ that~$g(l_n,n)\downarrow$.
    Hence, the test~$T_0,T_1,\dots$ has a measure exactly equal to $4+2cL_{comp}$, which is finite and computable. Moreover, for every of infinitely many~$n$ such that~$\beta \in S_n$, it still holds by~\eqref{eq:T_n-contains-alpha} that~$T_n$ is defined and contains~$\alpha$. Therefore, the test~$T_0,T_1,\dots$ is a total Solovay test that fails on~$\alpha$, hence~$\alpha$ is Schnorr nonrandom.
\end{enumerate}
\end{proof}


\section{Main result}\label{section:main-result}
As a main result of this paper, we claim that the reducibilities via $\mathbb{Q}$-translation functions imply the corresponding reducibilities via $\mathbb{R}$-translation functions.
\begin{theorem}\label{main-result}
    For arbitrary reals $\alpha$ and $\beta$, the following implications hold:
    \begin{equation}
        \alpha\redsolovay\beta \implies \alpha\redclopen\beta\makebox[6em]{and}\alpha\redsolovaytotal\beta \implies \alpha\redcllocal\beta.
    \end{equation}
    Moreover, if~$\alpha$ is not computable, then the following implication holds:
    \begin{equation}\label{eq:if-noncomputable}
        \alpha\redsolovay\beta \implies \alpha\redsolovayreal\beta.
    \end{equation}
\end{theorem}

\begin{proof}
    Let $g$ be a $\mathbb{Q}$-translation function from~$\beta$ to~$\alpha$, and let $\alpha$ and $\beta$ be two reals such that~$\alpha\redsolovay\beta$ with some constant~$c$ via~$g$. If $\beta$ is left-c.e., then~$\alpha$ is left-c.e.\ as well since the set of left-c.e.\ reals is $\redsolovay$-closed downwards, and the theorem statement holds by~\cite[Theorem~1]{Kumabe-etal-2020}. So, in what follows, suppose that~$\beta$ is not left-c.e., i.e.,~$\mathrm{LC}(\beta)$ is not recursively enumerable.

    Further, let~$b$ be the maximal real that fulfills the property
    \begin{equation}\label{eq:define-b}
        g(q)\downarrow\text{ for every }q < b.
    \end{equation}
    It holds obviously that~$b\geq \beta$ since~$g$ is a translation function from~$\beta$ to~$\alpha$, wherein, in case~$\alpha\redsolovaytotal\beta$, we obtain~$b=1$.
    
    In the scope of the proof, we set an enumeration~$q_0,q_1,\dots$ of the domain of~$g$. If there exists a rational~$b>\beta$ such that there is no~$q_i\in (\beta,b)$, then $\mathrm{LC}(\beta) = \mathrm{dom}(g)\cap [0,\beta)$ is recursively enumerable, a contradiction. Therefore, there exists a subsequence of~$q_0,q_1,\dots$ which tends to $\beta$ from above.
    
    \subsection*{Construction of a weakly $\mathbb{R}$-translation function from~$\beta$ to~$\alpha$}
    Now, we construct a Lipschitz continuous weakly $\mathbb{R}$-translation function $h$ from~$\beta$ to~$\alpha$ in four steps: first, on the base of~$g$, we construct another partial function~$\Tilde{g}$ on rationals that witnesses the Solovay reducibility $\alpha\redsolovay\beta$ with the same constant~$c$; second, on the base of~$\Tilde{g}$, we construct a technical computable two-argument partial function~$f(\cdot,\cdot)$ on rationals; third, on the base of~$f$, we construct another technical two-argument partial function~$\Tilde{f}(\cdot,\cdot)$ on rationals; and fourth, on the base of~$\Tilde{f}$, we finally construct the desired function~$h$.

    The lists of properties of the functions~$\Tilde{g}$ through~$h$ are combined in Claims~\ref{first-claim}--\ref{fourth-claim}, respectively. The proof of each claim can be easily obtained from the previous one through simple calculations.
    
    First, we define a computable function~$\Tilde{g}$ from rationals to rationals by setting
    \begin{equation*}
        Q_n= \{q_m : m\leq n\text{ and }q_m\leq q_n\} \makebox[3em]{and} \Tilde{g}(q_n) = \max\{g(q): q\in Q_n \}\text{ for all }n.
    \end{equation*}
    \begin{claimaaa}\label{first-claim}
        $\Tilde{g}$ satisfies the following properties:
        \begin{align}
            \label{eq:g-Tilde-g-domains-coincide}
            &\mathrm{dom}(\Tilde{g}) = \mathrm{dom}(g),
            \\
            \label{eq:Tilde-g-on-the-left-cut}
            &0 < \alpha - \Tilde{g}(q)\downarrow \leq \alpha - g(q) < c(\beta - q)\text{ for all }q\in \mathrm{LC}(\beta).
            \\
            \label{eq:Tilde-g-is-half-monotone}
            &Q_m\subseteq Q_n\text{ and }\Tilde{g}(q_m)\leq\Tilde{g}(q_n) \text{ for all }q_m,q_n\in \mathrm{dom}(\Tilde{g})\text{ where}\begin{cases}
                m<n,
                \\
                q_m<q_n.
            \end{cases}
        \end{align}
    \end{claimaaa}   
    
    From~\eqref{eq:g-Tilde-g-domains-coincide}, we obtain that, in particular, that~$q_0,q_1,\dots$ is at the same time a computable enumeration of~$\mathrm{dom}(\Tilde{g})$. The property~\eqref{eq:Tilde-g-on-the-left-cut} directly implies that the function~$g$ is a well-defined $\mathbb{Q}$-translation function that witnesses the Solovay reducibility of~$\alpha$ to~$\beta$ with the same constant as~$g$.
    
\medskip
    
    At the beginning of the second step, we fix a rational~$d$ and a natural~$K>1$ such that~$2^K = d>c$ and define a computable two-argument partial function $f:\mathbb{Q}\times\mathbb{N}\to\mathbb{Q}$ as follows:
    given $q$ and $n$, we enumerate $q_0(=0),q_1,\dots$ and, in case if we meet after some enumeration step $t$ a set $\{q_{i_0}(= 0), q_{i_1},\dots,q_{i_k}\}$ where~$i_0,i_1,\dots,i_k\leq t$ such that
    \begin{align}
        \label{eq:P(q,n)-from-0-to-k}
        &0<q_{i_n}-q_{i_{n-1}}<2^{-n}\text{ for all }i\in\{1,\dots,k\}\text{ and}
        \\
        \label{eq:P(q,n)-k}
        &0\leq q - q_{i_k} < 2^{-n},
    \end{align}
    let~$\{q_{i_0},q_{i_1},\dots,q_{i_N}\}$ be the greatest (under inclusion) such set of indices not larger than $t$.
    Such set exists since, if two different sets $A = (q_{i_0}, \dots, q_{i_a})$ and $B = (q_{j_0}, \dots, q_{j_b})$, where $i_0,\dots,i_a,j_0,\dots,j_b \leq t$, fulfill the properties~\eqref{eq:P(q,n)-from-0-to-k}--\eqref{eq:P(q,n)-k}, then the same holds for the set $A\cup B$. 
    
    Next, we define the ordered tuple
    \begin{equation*}
        P(q,n) = (q_{i_0},\dots,q_{i_k})\makebox[3em]{and} f(q,n) = \min\big\{\Tilde{g}(p)+d(q-p) : p \in P(q,n)\big\}.
    \end{equation*}
    Accordingly, we write $P(q,n)\uparrow$ and $f(q,n)\uparrow$ if the search for the set $P(q,n)$ never ends.
    
    Further, we define the (in case~$b=\beta$ coinciding) sets
    \begin{align*}
        &D_{\beta} = \{(q,n) : n\in\mathbb{N}\text{ and }q\in [0,\beta + 2^{-n})\}\text{ and}
        \\
        &D_b = \{(q,n) : n\in\mathbb{N}\text{ and }q\in [0,b + 2^{-n})\},
    \end{align*}
    that obviously fulfill~$D_{\beta}\subseteq D_b$.

    \begin{claimaaa}\label{second-claim}
        $f$ satisfies the following properties:    
        \begin{align}
            \label{eq:f-covers-0-b}
            &P(q,n)\downarrow\text{ and }f(q,n)\downarrow\text{ for all naturals }n\text{ and rationals }q\in [0,b),
            \\
            \label{eq:f-is-defined-on-D-b}
            &P(q,n)\downarrow\text{ and }f(q,n)\downarrow\text{ for all }(q,n)\in D_b,
            \\
            \label{eq:f-is-not-so-far-from-alpha}
            &f(q,n)\downarrow<\alpha + 2^{-(n-K-1)}\text{ for all }(q,n)\in D_{\beta},
            \\
            \label{eq:P-subset}
            &P(p,n)\downarrow\subseteq P(q,n)\text{ for all }(q,n)\in \mathrm{dom}(f)\text{ and }p<q,
            \\
            \label{eq:L-second-argument}
            &\begin{cases}
                P(q,m)\downarrow\subseteq P(q,n)
                \\
                0\leq f(q,m) - f(q,n)\downarrow<2^{-(m-K)}
            \end{cases}\text{for all }(q,n)\in \mathrm{dom}(f)\text{ and }m<n,
            \\
            \label{eq:L-first-argument}
            &f(q,n) - f(p,n)\downarrow< d(q - p)\text{ for all }(q,n)\in \mathrm{dom}(f)\text{ and }p<q,
            \\
            \label{eq:f-has-Solovay-property}
            &\alpha - f(q,n)\downarrow < c(\beta - q)\text{ for every }n\in\mathbb{N}\text{ and }q\in \mathrm{LC}(\beta).
        \end{align}
    \end{claimaaa}

    \medskip

    In the second step, we define a computable two-argument partial function $\Tilde{f}:\mathbb{Q}\times\mathbb{N}\to\mathbb{Q}$ by setting 
    \begin{equation*}
        \Tilde{f}(q,n) = \max\big(\{f(q',n): q'\in P(q,n)\}\cup \{f(q,n)\}\big)
    \end{equation*}
    for all~$(q,n)\in \mathrm{dom}(f)$.
    
    \begin{claimaaa}\label{third-claim}
    $\Tilde{f}$ satisfies the following properties:
        \begin{align}
            \label{eq:f-Tilde-f-domains-coincide}
            &\mathrm{dom}(\Tilde{f}) = \mathrm{dom}(f),
            \\
            \label{eq:Tilde-f>f}
            &\Tilde{f}(q,n)\geq f(q,n)\text{ for all }(q,n)\in\mathrm{dom}(\Tilde{f}),
            \\
            \label{eq:Tilde-f-is-not-so-far-from-alpha}
            &\Tilde{f}(q,n)\downarrow < \alpha + 2^{-(n-K-1)}\text{ for all }(q,n)\in D_{\beta},
            \\
            \label{eq:Tilde-L-second-argument}
            &|\Tilde{f}(q,n) - \Tilde{f}(q,m)| < 2^{-(m-K-1)}\text{ for all }(q,m)\in \mathrm{dom}(f)\text{ and }m<n,
            \\
            \label{eq:Tilde-L-first-argument}
            &|\Tilde{f}(q,n) - \Tilde{f}(p,n)\downarrow| < 2^{-(n-K)}+2^K|q-p|\text{ for all }(q,n)\in\mathrm{dom}(f)\text{ and }p<q,
            \\
            \label{eq:Tilde-f-has-Solovay-property}
            &\alpha - \Tilde{f}(q,n)\downarrow < c(\beta - q)\text{ for every }n\in\mathbb{N}\text{ and }q\in \mathrm{LC}(\beta).
        \end{align}
    \end{claimaaa}
    
    \medskip
    
    In the third step, we define a real function
    \begin{equation*}
        h(x) =\lim\limits_{n\to\infty}f(q_n,n)\text{ for every effective approximation }q_n\underset{n\to\infty}{\to}x.
    \end{equation*}
    
    \begin{claimaaa}\label{fourth-claim}
    $h$ satisfies the following properties:
        \begin{align}
            \label{eq:h-is-well-defined}
            &[0,b)\in \mathrm{dom}(h);
            \\
            \label{eq:h-L}
            &h\text{ is Lipschitz continuous};
            \\
            \label{eq:h-left-to-left}
            &h(x)\leq \alpha\text{ for every }x<\beta;
            \\
            \label{eq:h-cofinal}
            &\lim\limits_{\substack{x\to\beta \\ x \in [0,b)}}h(x) = \alpha;
            \\
            \label{eq:Type-2-computablility}
            &h\text{ is (totally Type 2) computable on }(0,b).
        \end{align}
    \end{claimaaa}

\subsection*{The constructed weakly $\mathbb{R}$-translation function demonstrates the theorem statement}

By~\eqref{eq:h-is-well-defined} through~\eqref{eq:Type-2-computablility}, we obtain that $h$ witnesses $\alpha\redclopen\beta$ and, in case $b=1$, even $\alpha\redcllocal\beta$.

\medskip

Now, in case $\alpha$ is not computable, we even obtain that~$h(x) <\alpha$ for every~$x<\beta$ by contradiction:
assuming~$h(\Tilde{x})= \alpha$ for some real~$\Tilde{x}<\beta$, we can fix two rationals~$\Tilde{p}$ and~$\Tilde{q}$ such that $\Tilde{x}< \Tilde{p} < \Tilde{q} < \beta$. Then function~$\Tilde{h}$ defined on the compact interval~$[\Tilde{p},\Tilde{q}]$ by 
\[\Tilde{h}(x) = \max\{h(y): y\in [0,x]\}\]
is computable on the whole interval~$[\Tilde{p},\Tilde{q}]$ by~\cite[Corollary~6.2.5]{Weihrauch-2000}.
On the other hand, by~\eqref{eq:h-left-to-left}, we have~$h(y)\leq \alpha$ for every~$y\in [0,\beta)$, hence our assumption implies for every $x\in [\Tilde{p},\Tilde{q}]$ that
$\Tilde{h}(x) = \max\{h(y): y\in [0,x]\} = h(\Tilde{x}) = \alpha$;
thus, the function~$\Tilde{h}$ is a constant function defined on~$[\Tilde{p},\Tilde{q}]$ that returns~$\alpha$ for every input. Therefore, $\alpha$ should be computable as a limit point of an effective approximation returned by a Turing machine that computes $\Tilde{h}$ from the (computable) input~$(\frac{1}{2}(\Tilde{p}+\Tilde{q}),\frac{1}{2}(\Tilde{p}+\Tilde{q}),\dots)$. A contradiction.

Thus, $h$ is an $\mathbb{R}$-translation function from $\beta$ to $\alpha$, and we obtain from its Lipschitz continuity that $\alpha\redsolovayreal\beta$. 

\end{proof}

On the other hand, by Proposition~\ref{eq:1-solovay-reducible-to-nonleft-ce}, there exists a nonleft-c.e.\ real~$\beta$ such that $1$ is Solovay reducible to it, while, by Corollary~\ref{leftce-standalone-real}, as a computable real, $1$ cannot be real Solovay reducible to~$\beta$. These two observations imply together the following result.
\begin{proposition}
    There exists a computable real~$\alpha$ and a nonleft-c.e.\ real~$\beta$ such that
    $\alpha\redsolovay\beta$ and $\alpha\notredsolovayreal\beta.$
\end{proposition}

\medskip

The implication~$\redsolovay\implies\redclopen$ is strict on~$\mathbb{R}$ since computable reals form the least degree on~$\mathbb{R}$ relative to~$\redclopen$ by Proposition~\ref{computable-form-least-degree} but not relative to~$\redsolovay$ by \cite[Proposition~9.6.1]{Downey-Hirschfeldt-2010}.
We still do not know whether the implication~$\redsolovay\implies\redsolovayreal$ is strict on noncomputable reals.

The implication~$\redsolovaytotal\implies\redcllocal$ is strict since computable reals form the least degree on~$\mathbb{R}$ (and thus also on the set of left-c.e.\ reals) relative to~$\redcllocal$ by Proposition~\ref{computable-form-least-degree} but not relative to~$\redsolovaytotal$ by \cite[Proposition~2.5]{Merkle-Titov-2022}.

\section*{Acknowledgments}
I would like to thank Peter Hertling, Wolfgang Merkle, and Laurent Bienvenu for meticulous proofreading of this article, as well as Subin Pulari, Andrei Romashchenko, and Alexander Shen for attending an unofficial presentation and fixing some logical errors.

\end{document}